\documentclass{amsart}%
\usepackage{amsfonts}
\usepackage{amsmath}
\usepackage{amssymb}
\usepackage{graphicx}
\usepackage{latexsym}
\usepackage{amstext}
\usepackage{eucal}
\usepackage{verbatim}%
\providecommand{\U}[1]{\protect\rule{.1in}{.1in}}
\newtheorem{theorem}{Theorem}
\theoremstyle{plain}

\newtheorem{corollary}{Corollary}

\newtheorem{definition}{Definition}
\newtheorem{example}{Example}

\newtheorem{lemma}{Lemma}

\newtheorem{remark}{Remark}

\numberwithin{equation}{section}
\begin{document}
\title[XCF on a negatively curved solid torus]{Cross curvature flow on a\\negatively curved solid torus}
\author{Jason Deblois}
\address[Jason Deblois]{ University of Illinois at Chicago}
\email{jdeblois@math.uic.edu}
\urladdr{http://www.math.uic.edu/\symbol{126}jdeblois/}
\author{Dan Knopf}
\address[Dan Knopf]{ University of Texas at Austin}
\email{danknopf@math.utexas.edu}
\urladdr{http://www.ma.utexas.edu/users/danknopf}
\author{Andrea Young}
\address[Andrea Young]{ University of Arizona}
\email{ayoung@math.arizona.edu}
\urladdr{http://math.arizona.edu/\symbol{126}ayoung}
\thanks{D.K.~ acknowledges NSF support in the form of grants DMS-0505920 and DMS-0545984.}

\begin{abstract}
The classic $2\pi$-Theorem of Gromov and Thurston constructs a negatively
curved metric on certain $3$-manifolds obtained by Dehn filling. By
Geometrization, any such manifold admits a hyperbolic metric. We outline a
program using cross curvature flow to construct a smooth one-parameter family
of metrics between the \textquotedblleft$2\pi$-metric\textquotedblright\ and
the hyperbolic metric. We make partial progress in the program, proving
long-time existence, preservation of negative sectional curvature, curvature
bounds, and integral convergence to hyperbolic for the metrics under consideration.

\end{abstract}
\maketitle

\section{Introduction}

In this note, we outline a program that uses cross curvature flow to answer
certain questions inspired by the $2\pi$-Theorem of Gromov and Thurston. We
begin and make partial progress toward completing this program. More
specifically, we consider cross curvature flow on a class of negatively curved
metrics on the solid torus, the simplest nontrivial handlebody. This is motivated 
by the \textquotedblleft Dehn surgery\textquotedblright\ construction in
3-manifold topology, of which we give a brief account below. We apply the flow
to a negatively curved metric described by Gromov--Thurston on a solid torus
with prescribed Dirichlet boundary conditions.

Perelman's use of Ricci flow with surgery to prove the Geometrization
Conjecture demonstrates the considerable power of geometric flows to address
questions about $3$-manifolds \cite{Pe1, Pe2}. Subsequent work, that of
Agol--Storm--Thurston \cite{AST}, for example, shows that Ricci flow may give
information even about $3$-manifolds known \textit{a priori} to admit a metric
of constant curvature. The results of \cite{AST} are obtained by using the
monotonicity of volume under Ricci flow with surgery.

Ricci flow is nonetheless an imperfect tool for analyzing \textit{hyperbolic}
$3$-manifolds --- those admitting metrics with constant negative sectional
curvatures. According to Geometrization, such manifolds form the largest class
of prime $3$-manifolds. Paradoxically, they are also the least understood
class. In particular, the facts that Ricci flow does not preserve negative
curvature and that singularities must be resolved by surgery leaves certain
basic questions unanswered. For example, one may ask whether the space of
negatively curved metrics on a manifold $M$ is connected or contractible, or
analogous questions for the space of metrics with negative curvature
satisfying various pinching conditions. In higher dimensions, these questions
are known to have negative answers \cite{FO1, FO2, FO3}; but it has been
conjectured and seems plausible that their answers should be positive if $M$
is a hyperbolic $3$-manifold.

Such questions have particular relevance to the study of hyperbolic
$3$-manifolds, because in practice one frequently encounters $3$-manifolds on
which it is possible explicitly to describe a metric of nonconstant negative
curvature and to discern certain geometric information; however, the
relationship between this metric and the hyperbolic metric --- which exists by
Geometrization --- is unclear. Such metrics are constructed by Namazi and
Souto in \cite{NaS}, for example, as well as in the famous $2\pi$-Theorem\ of
Gromov--Thurston \cite{GroTh}, which motivates our work here.

The $2\pi$-Theorem belongs a family of results describing the geometry of
\textquotedblleft fillings\textquotedblright\ --- that is, manifolds which
result from \textit{Dehn surgery} on cusps of a finite-volume hyperbolic
$3$-manifold. We describe this construction in greater detail in
\S \ref{Dehn}. The initial such result, the \textquotedblleft hyperbolic Dehn
surgery theorem,\textquotedblright\ \cite{Th1} asserts that most fillings of a
hyperbolic $3$-manifold admit hyperbolic metrics of their own, but it does not
explicitly describe the set of hyperbolic fillings. The $2\pi$-Theorem
remedies this deficiency, but with a weaker conclusion, supplying a metric
with negative sectional curvatures on a filling that satisfies a certain
explicit criterion. Another result in the spirit of the $2\pi$-Theorem is the
$6$-Theorem due to Agol \cite{Agol} and, independently, Lackenby \cite{Lack}.
This relaxes the hypotheses of the $2\pi$-Theorem at the expense of weakening
its conclusion to a purely topological statement, one that nonetheless implies
hyperbolizability using Geometrization.

A program more in the spirit of this paper is that of Hodgson--Kerckhoff
\cite{HK1, HK2, HK3}, which, under somewhat stronger hypotheses, describes a
family of \textquotedblleft cone\textquotedblright\ metrics. These are
hyperbolic but singular along an embedded one-manifold, and interpolate
between the original smooth hyperbolic metric and that on the filling. The
goal of our program is to give an alternate construction of an interpolating
family. We propose using \textit{cross curvature flow} (\textsc{xcf}) to
construct a one-parameter family of smooth metrics, with negative but
nonconstant sectional curvatures, which interpolate between the metric
supplied by the $2\pi$-Theorem and the hyperbolic metric guaranteed by Geometrization.

In \cite{ChowHam}, Chow and Hamilton introduced the \textsc{xcf} of metrics on
a $3$-manifold with uniformly signed sectional curvatures. They established
certain properties of the flow and conjectured that if the initial datum is a
metric with strictly negative sectional curvatures, then \textsc{xcf}
(suitably normalized) should exist for all time, preserve negative sectional
curvature, and converge asymptotically to a metric of constant curvature. This
conjecture would imply contractibility for the space of negatively curved
metrics on a hyperbolic $3$-manifold. An additional benefit of a proof of this
conjecture is that it should be easier to track the evolution of relevant
geometric quantities in a flow without singularities than it is to track their
evolution across surgeries.

\textsc{xcf} is a fully nonlinear, weakly parabolic system of equations, which
can be defined as follows: let $P_{ab}=R_{ab}-\frac{1}{2}Rg_{ab}$ denote the
Einstein tensor, and let $P^{ij}=g^{ia}g^{jb}P_{ab}$. One can define the cross
curvature tensor $X$ by
\begin{equation}
X_{ij}=\frac{1}{2}P^{uv}R_{iuvj}.
\end{equation}
Then the cross curvature flow of $(M^{3},g_{0})$ is%
\begin{align}
&  \frac{\partial}{\partial t}g=-2X,\label{xcf}\\
&  g(x,0)=g_{0}(x).
\end{align}
Since its introduction in \cite{ChowHam}, several papers concerning
\textsc{xcf} have appeared in the literature. Buckland established short-time
existence of \textsc{xcf} for smooth initial data on compact manifolds using
DeTurck diffeomorphisms to create a strictly parabolic system \cite{Buck}.
(Like Ricci flow, \textsc{xcf} is only weakly parabolic.) Chen and Ma showed
that certain warped product metrics on $2$-torus and $2$-sphere bundles over
the circle are solutions to \textsc{xcf} \cite{MaCh}. Solutions on
locally-homogeneous manifolds have recently been studied by Cao, Ni, and
Saloff-Coste \cite{CNS, CS}, and independently by Glickenstein \cite{Glick}.
Two of the authors of this paper have shown that (normalized) \textsc{xcf} is
asymptotically stable at hyperbolic metrics \cite{KnYo}. Also, in unpublished
earlier work, Andrews has obtained interesting estimates for more general
solutions of \textsc{xcf}.

Here we consider a family of negatively curved metrics on the solid torus
$D^{2}\times S^{1}$, with initial data that arise in the proof of the $2\pi
$-Theorem. We (1) show short-time existence of \textsc{xcf} with prescribed
boundary conditions for metrics in this family, (2) establish curvature bounds
depending only on the initial data (showing in particular that negative
curvature is preserved for this family), and (3) demonstrate long-time existence.

Although they fall short of confirming the Chow--Hamilton conjectures (even
for metrics in this family) our results do provide evidence in favor of those
conjectures. On one hand, it is encouraging that negative curvature is
preserved and long--time existence holds. On the other hand, the fully
nonlinear character of \textsc{xcf} and the high topological complexity of
negatively curved manifolds have not yet allowed a complete proof of
convergence to hyperbolic. In fact, we are not yet able to rule out
convergence to a soliton. Note that recent studies of \textsc{xcf} (properly
defined) on homogeneous manifolds admitting curvatures of mixed sign display a
much greater variety of behaviors than the corresponding examples for Ricci
flow (cf.~\cite{CNS, CS}, \cite{Glick}).

The purpose of this paper is to introduce our program and to report on partial
progress toward its resolution. The paper is organized as follows. In
\S \ref{Dehn}, we give a brief description of the $2\pi$-Theorem and related
work. In \S \ref{Description}, we describe the proposed program and give some
of its motivations. In \S \ref{Setup}, we derive relevant curvature equations
and establish appropriate boundary conditions. In \S \ref{Summary}, we report
on our progress thus far, summarizing the new results established in this
paper that support the program. To streamline the exposition, the (sometimes
lengthy) proofs of these theorems are collected in \S \ref{CollectedProofs}.

We hope that the results in this paper contribute to further progress in the
program, and ultimately to further applications of curvature flows to the
study of hyperbolic $3$-manifolds.

\section{Dehn filling and the $2\pi$-Theorem\label{Dehn}}

Suppose $M$ is a compact $3$-manifold with a single boundary component: a
torus $T^{2}\cong S^{1}\times S^{1}$. A closed manifold may be produced from
$M$ by \textit{Dehn filling} $\partial M$, a well known construction in which
a solid torus is glued to $M$ by a homeomorphism of their boundaries. More precisely:

\begin{definition}
Suppose $M$ is a compact $3$-manifold with a torus boundary component $T$. Let
$D^{2}$ be the unit disk in $\mathbb{R}^{2}$; let $h$ be a homeomorphism from
$\partial(D^{2}\times S^{1})=\partial D^{2}\times S^{1}$ to $T$; and for $p\in
S^{1}$, take $\lambda=h(\partial D^{2}\times\{p\})$. Define the manifold
$M(\lambda)$ obtained by \textit{Dehn filling }$M$\textit{ along }$\lambda$
by:
\[
M(\lambda)=\frac{M\sqcup(D^{2}\times S^{1})}{x\sim h(x),~x\in\partial
D^{2}\times S^{1}}.
\]
We call the isotopy class of $\lambda$ in $T$ the \emph{filling slope}, and
the image in $M(\lambda)$ of $D^{2}\times S^{1}$ the \emph{filling torus.}
\end{definition}

Clearly, the topology of $M(\lambda)$ may vary depending on the choice of $h$,
but it is a standard fact that it is determined up to homeomorphism by the
choice of filling slope.

Now suppose that $M$ is a hyperbolic $3$-manifold with finite volume. (That
is, $M$ is a smooth manifold equipped with a complete, finite-volume
Riemannian metric which has sectional curvatures $\equiv-1$.) It follows from
the Margulis Lemma (\cite{Margulis}, cf.~\cite[Theorem~12.5.1]{Rat}) that if
$M$ is noncompact, each of its \textit{cusps} is diffeomorphic to $T^{2}%
\times\mathbb{R}^{+}$. (Here a \textquotedblleft cusp\textquotedblright\ is a
component of the complement of a sufficiently large compact
submanifold-with-boundary.) Thus removing the cusps of $M$ yields a compact
$3$-manifold $\widehat{M}$ whose boundary is a disjoint union of tori.

Suppose for simplicity that $M$ has a single cusp. The hyperbolic Dehn surgery
Theorem, due to Thurston \cite[\S 5]{Th1}, asserts that for all but finitely
many choices of slope $\lambda$ on $\partial\widehat{M}$, the closed manifold
$\widehat{M}(\lambda)$ admits a hyperbolic metric. In the past thirty years,
an enormous quantity of work has been devoted to explicitly describing the set
of \textquotedblleft hyperbolic filling slopes\textquotedblright, and
investigating the relationship between the hyperbolic metrics on $M$ and
$\widehat{M}(\lambda)$. This is the context of the $2\pi$-Theorem, due to
Gromov and Thurston (\cite{GroTh}, cf.~\cite[Theorem~9]{BH}).

A further consequence of the Margulis Lemma is that a parameterization
$\phi\colon\,T^{2}\times\mathbb{R}^{+}\rightarrow M$ for the cusp of $M$ may
be chosen so that the hyperbolic metric on $M$ pulls back to a warped product
metric on $T^{2}\times\mathbb{R}^{+}$ which restricts on each level torus
$T^{2}\times\{t\}$ to a Euclidean metric. (We will precisely describe this
metric below.) We call such a parameterization \textquotedblleft
good\textquotedblright.

\begin{theorem}
[Gromov--Thurston]\label{GT2pi}Let $M$ be a one-cusped hyperbolic $3$-manifold
with a good parameterization $\phi\colon\,T^{2}\times\mathbb{R}^{+}\rightarrow
M$ of the cusp of $M$, and suppose $\lambda$ is a slope on $\partial
\widehat{M}$ such that $\phi^{-1}(\lambda)$ has a geodesic representative with
length greater than $2\pi$ in some level torus. Then $\widehat{M}(\lambda)$
admits a Riemannian metric with negative sectional curvatures.
\end{theorem}

The proof is constructive, yielding a metric on $\widehat{M}(\lambda)$ which
on $\widehat{M}$ is isometric to the metric inherited from $M$, and on the
filling torus has variable negative curvature. Below we will sketch a proof,
describing the metrics under consideration on $D^{2}\times S^{1}$. It will be
convenient to use \textit{cylindrical coordinates} $(\mu,\lambda,r)\in
\lbrack0,1)\times\lbrack0,1)\times\lbrack0,1]$, where with $D^{2}\times S^{1}$
naturally embedded in $\mathbb{R}^{2}\times\mathbb{R}^{2}$, we have
\begin{align*}
x_{1} &  =r\cos(2\pi\mu), &  &  x_{2}=\cos(2\pi\lambda),\\
y_{1} &  =r\sin(2\pi\mu), &  &  y_{2}=\sin(2\pi\lambda).
\end{align*}
The \textit{core} of $D^{2}\times S^{1}$ is the set $\{(0,0)\}\times
S^{1}=\{(0,\lambda,0)\}$, and a \textit{meridian disk} is of the form
$D^{2}\times\{p\}=\{(\mu,\lambda_{0},r)\}$ for fixed $p$ or $\lambda_{0}$. The
standard rotations of $D^{2}\times S^{1}$ are of the form $(\mu,\lambda
,r)\mapsto(\mu+\mu_{0},\lambda,r)$, fixing the core and rotating each meridian
disk, or $(\mu,\lambda,r)\mapsto(\mu,\lambda+\lambda_{0},r)$. All of the
metrics we use will be \textit{rotationally symmetric}, and \textit{diagonal}
in cylindrical coordinates, which implies that they have the form
\begin{equation}
G(r)=f^{2}(r)\,d\mu^{2}+g^{2}(r)\,d\lambda^{2}+h^{2}(r)\,dr^{2}.
\end{equation}
Here the functions $f$, $g$, and $h$ must satisfy the following regularity
conditions at $r=0$ to ensure that they extend smoothly across the core.

\begin{lemma}
$G$ extends to a smooth metric on $D^{2}\times S^{1}$ if and only if $g(r)$
and $h(r)$ extend to smooth even functions on $\mathbb{R}$ with $g(0),h(0)>0$,
and $f$ extends to a smooth odd function on $\mathbb{R}$ with $f_{r}(0)={2\pi
}h(0)$.
\end{lemma}

\noindent Because this is a standard result in Riemannian geometry, we omit
the proof.

\medskip

In a rotationally-symmetric diagonal metric on $D^{2}\times S^{1}$, it will
frequently prove convenient to parameterize by an outward-pointing radial
coordinate which measures distance from the core:
\begin{equation}
s(r)=\int_{0}^{r}h(\rho)\,\mathrm{d\rho}. \label{DefineArcLength}%
\end{equation}
Then $G$ may be written in the form
\begin{equation}
G(s)=f^{2}(s)\,d\mu^{2}+g^{2}(s)\,d\lambda^{2}+ds^{2},\qquad s\in
\lbrack0,s_{0}]. \label{metric}%
\end{equation}
Here $s_{0}$ is the distance from the core to the boundary. The smoothness
criteria of the lemma above translate to the requirements that $f$ extends to
a smooth odd function of $s$, and $g$ to an even function of $s$, such that
$f_{s}(0)={2\pi}$ and $g(0)>0$.

\begin{example}
There is a standard description of the cusp of a hyperbolic manifold as the
quotient of a \textit{horoball} in hyperbolic space $\mathbb{H}^{3}$ by a
group of isometries $\Gamma\simeq\mathbb{Z}^{2}$. We will use the upper
half-space model for $\mathbb{H}^{3}$, namely $\{(x,y,z)\in\mathbb{R}%
^{3}\,|\,z>0\}$ with the Riemannian metric $G_{h}=z^{-2}(dx^{2}+dy^{2}%
+dz^{2})$. Then a standard horoball is $\mathcal{H}_{1}=\{(x,y,z)\,|\,z\geq
1\}$, and the elements of $\Gamma$ are Euclidean translations fixing the $z$
coordinate, each of the form $\phi_{(a,b)}:(x,y,z)\mapsto(x+a,y+b,z)$ for
$(a,b)\in\mathbb{R}^{2}$.

Without loss of generality, we may assume that $\Gamma$ is generated by
elements of the form $\phi_{(M,0)}$ and $\phi_{(x_{0},L)}$, where $M,L>0$.
Then $M$ is the length of a \textquotedblleft meridian\textquotedblright%
\ geodesic of the torus $T=\partial(\mathcal{H}_{1}/\Gamma)=\{(x,y,1)\}/\Gamma
$, and $L$ is the length of an arc joining the meridian geodesic to itself,
perpendicular to it at both ends. (The final constant $x_{0}$ of the
definition is a sort of \textquotedblleft twisting parameter\textquotedblright%
\ and does not enter into computations using the metrics under consideration here.)

There is a map from the complement of the core in $D^{2}\times S^{1}$ to
$\mathcal{H}_{1}/\Gamma$, given in cylindrical coordinates by
\[
\Psi(\mu,\lambda,r)=\left(  M\mu,L\lambda,\frac{1}{r}\right)  ,\qquad
r\in(0,1].
\]
Using $\Psi$, the hyperbolic metric pulls back to
\[
\Psi^{\ast}G_{h}=M^{2}r^{2}\,d\mu^{2}+L^{2}r^{2}\,d\lambda^{2}+\frac{dr^{2}%
}{r^{2}},\qquad r\in(0,1].
\]
Using an outward-pointing radial coordinate which measures distance to the
boundary,
\[
s(r)=-\int_{r}^{1}\frac{1}{\rho}\,\mathrm{d\rho}\,=\log r,
\]
the hyperbolic metric takes the form
\[
G_{\mathrm{cusp}}=M^{2}e^{2s}\,d\mu^{2}+L^{2}e^{2s}\,d\lambda^{2}%
+ds^{2},\qquad s\in(-\infty,0].
\]

\end{example}

\begin{proof}
[Sketch of proof of Theorem~\ref{GT2pi}]One seeks a rotationally symmetric
diagonal metric on $D^{2}\times S^{1}$ of the form (\ref{metric}), which in
addition has negative sectional curvatures. One further requires that in a
neighborhood $U$ of $\partial(D^{2}\times S^{1})$, $f$ and $g$ take the form
\begin{equation}
f(s)=Me^{s-s_{0}},\qquad g(s)=Le^{s-s_{0}}. \label{cusp metric}%
\end{equation}
If this is the case, then the map $h:(\mu,\lambda,s)\mapsto(\mu,\lambda
,s-s_{0})$ takes $U$ isometrically to a neighborhood of $\partial
(\mathcal{H}_{1}/\Gamma)$ with the metric $G_{\mathrm{cusp}}$. Note also that
for any $p\in S^{1}$, $\lambda=h(\partial D^{2}\times\{p\})$ is a geodesic on
$T$ with length $M$; thus in this case, there is a metric on $\widehat
{M}(\lambda)$ yielding the theorem. A computation establishes that the
sectional curvatures of $G(s)$ are as follows:
\[
K_{\lambda\mu}=-\frac{f_{s}g_{s}}{fg},\qquad K_{\lambda s}=-\frac{g_{ss}}%
{g},\qquad K_{\mu s}=-\frac{f_{ss}}{f}.
\]
Since $f$ and $g$ are positive on $(0,s_{0}]$, and $f_{s}(0)=2\pi$, it follows
that all sectional curvatures are negative if and only if $f_{s}$, $g_{s}$,
$f_{ss}$, and $g_{ss}$ are positive on $(0,s_{0}]$. Since the components of
$G$ must satisfy (\ref{cusp metric}) near the boundary, one has $f_{s}%
(s_{0})=M$. Thus since $f_{ss}>0$, $M$ must be greater than $f_{s}(0)=2\pi$.
One can show that this necessary condition is also sufficient for there to
exist a metric of the form (\ref{metric}) satisfying our conditions. (See
\cite[Theorem 9]{BH}.)
\end{proof}

\section{Description of the program\label{Description}}

We now outline our proposed program. As was discussed above, the $2\pi
$-Theorem constructs a negatively curved metric on certain $3$-manifolds
obtained by Dehn filling. By Geometrization, any such manifold admits a
hyperbolic metric. Our goal is to use \textsc{xcf} to construct a smooth
one-parameter family of metrics between the \textquotedblleft$2\pi
$-metric\textquotedblright\ and the hyperbolic metric. Motivated by the
Hodgson--Kerckhoff constructions \cite{HK1, HK2, HK3}, one might expect the
program to work for slopes $\lambda$ on $\partial\widehat{M}$ such that
$\phi^{-1}(\lambda)$ has a geodesic representative of length $\geq L$, for
some $L>2\pi$. (Because our method depends on a stability result, one does not
expect to be able to take $L=2\pi$.)

The first step is to take as an initial datum a negatively curved metric
$G(s,0)$ on a solid torus given by equation~(\ref{metric}). We evolve this
metric by cross curvature flow. Because this is a manifold-with-boundary, we
impose time dependent Dirichlet boundary conditions corresponding to a
homothetically evolving hyperbolic metric. (Example~\ref{GeodesicTube} below
motivates this choice.) In this paper, we show that the flow exists for all
time and preserves negative sectional curvatures. Inspired by the conjecture
of Hamilton and Chow, we conjecture that (after normalization) $G(\cdot,t)$
converges uniformly to a hyperbolic metric in $C^{2,\alpha}$ as $t\rightarrow
\infty$, at least on sets compactly contained in the complement of the
boundary, i.e.~on $[0,1)\times\lbrack0,1)\times\lbrack0,r_{\ast}]$ for any
$r_{\ast}\in(0,1)$. The results in this paper support this conjecture.

Once the conjecture is proved, the second step of the program is to
homothetically evolve the negatively curved metric on $\widehat{M}(\lambda)$
that is guaranteed by the $2\pi$-Theorem. Recall that the metric on
$\widehat{M}$ is isometric to the hyperbolic metric inherited from $M$. The
Dirichlet boundary conditions on the solid torus are deliberately chosen to
match the boundary conditions on $\widehat{M}$. Once one has shown that the
gluing map is $\mathcal{C}^{2+\alpha}$, one may apply the asymptotic stability
of hyperbolic metrics under cross curvature flow (proved by two of the authors
\cite{KnYo}) to establish exponential convergence to a metric of constant
negative curvature.

A successful completion of this program would exhibit a one-parameter family
of negatively curved metrics connecting the \textquotedblleft$2\pi$
metric\textquotedblright\ on $\widehat{M}(\lambda)$ to the hyperbolic metric
that, according to Geometrization, this manifold must admit.

\section{Basic equations and boundary conditions\label{Setup}}

Let $M^{3}=D^{2}\times S^{1}$ be the solid torus, a $3$-manifold with boundary. 
As above, we restrict our attention to evolving metrics $G$ on $\mathcal{M}^{3}$ 
which are \emph{rotationally symmetric.} These are metrics that are diagonal in
cylindrical coordinates, and for which the isometry group acts transitively on
each torus consisting of the locus of points a fixed distance from the core.
We may write\ $G$ as
\begin{subequations}
\label{G}%
\begin{align}
G(s,t)  &  =f^{2}(s,t)\,d\mu^{2}+g^{2}(s,t)\,d\lambda^{2}+ds^{2}\\
&  =e^{2u(s,t)}\,d\mu^{2}+e^{2v(s,t)}\,d\lambda^{2}+ds^{2},
\end{align}
where $s$ is arclength measured outward from the core, as defined in
(\ref{DefineArcLength}).

The following example describes the prototypical negatively curved diagonal
metric on a solid torus: the hyperbolic metric of constant curvature $-1$.
This explains our choice of boundary conditions.
\end{subequations}
\begin{example}
\label{GeodesicTube}The metric on a regular neighborhood of a closed geodesic
in a hyperbolic manifold is given by
\[
G_{\mathrm{hyp}}=\left(  \frac{2}{1-r^{2}}\right)  ^{2}\left(  r^{2}\,d\mu
^{2}+\frac{b}{4}(1+r^{2})^{2}\,d\lambda^{2}+dr^{2}\right)
\]
for $r\in(0,1)$, where $b$ is the length of the geodesic. Using the geodesic
radial coordinate
\[
s(r)=\int_{0}^{r}\frac{2}{1-\rho^{2}}\,\mathrm{d\rho}\,=\log\left(  \frac
{1+r}{1-r}\right)  ,
\]
this takes the form
\[
G_{\mathrm{hyp}}=4\pi\sinh^{2}s\ d\mu^{2}+b\cosh^{2}s\ d\lambda^{2}+ds^{2}.
\]
The solution of \textsc{xcf} on any closed hyperbolic manifold with initial
data $G_{0}$ is $G(t)=\sqrt{1+4t}\,G_{0}$. On the solid torus, the solution of
\textsc{xcf} with initial data $G_{\mathrm{hyp}}$ is similarly $G(t)=\sqrt
{1+4t}\,G_{\mathrm{hyp}}$.
\end{example}

The principal sectional curvatures of the metric $G$ given by (\ref{G}) are%
\begin{subequations}
\begin{align}
K_{\lambda\mu}  &  =-\frac{f_{s}g_{s}}{fg}=-u_{s}v_{s},\\
K_{\lambda r}  &  =-\frac{g_{ss}}{g}=-(v_{ss}+v_{s}^{2}),\\
K_{\mu r}  &  =-\frac{f_{ss}}{f}=-(u_{ss}+u_{s}^{2}).
\end{align}

We are interested in the case that $G$ has negative sectional curvatures. This
occurs when $f$ and $g$, in addition to being positive, are monotonically
increasing convex functions of $s$ that satisfy the following conditions at
the origin:
\end{subequations}
\begin{subequations}
\label{core}%
\begin{align}
\lim_{s\rightarrow0}\frac{g_{s}}{f}  &  =\lim_{s\rightarrow0}\frac{g_{ss}%
}{f_{s}}=\lim_{s\rightarrow0}g_{ss}>0,\\
\lim_{s\rightarrow0}\frac{f_{ss}}{f}  &  =\lim_{s\rightarrow0}\frac{f_{sss}%
}{f_{s}}=\lim_{s\rightarrow0}f_{sss}>0.
\end{align}

One finds from the first equality in (\ref{core}) that the sectional
curvatures $K_{\mu\lambda}$ and $K_{r\lambda}$ are identical at the core. This
reflects the fact that rotation about the core is an isometry, and so the
sectional curvatures tangent to any two planes containing $\frac{\partial
}{\partial\lambda}$ are equal.

Henceforth, we assume the initial metric $G_{0}$ has negative sectional
curvatures. We denote the absolute values of the principal sectional
curvatures by $\alpha$, $\beta$, and $\gamma$, namely%
\end{subequations}
\begin{equation}
\alpha=-K_{\lambda\mu},\qquad\beta=-K_{\lambda r},\qquad\gamma=-K_{\mu r}.
\end{equation}

The assumption of negative sectional curvature imposes certain constraints on
the initial values of $f$, $g$ and $h$, as mentioned above. Namely, we require
that $f(1,0)=\ell_{1}>2\pi$ and that the initial radius $r_{0}=s(1,0)$
satisfies $1<r_{0}<\ell_{1}/2\pi$. We also assume that the metric
$G_{0}=G(\cdot,0)$ obeys a global $C^{2+\theta}$ bound and that $G_{0}$ has
constant negative sectional curvature at the core --- namely, that
$\alpha=\beta=\gamma>0$ at $r=0$.

We now apply cross curvature flow to a metric of the form (\ref{G}). The
evolution of this metric under \textsc{xcf} is equivalent to the system
\begin{subequations}
\label{xcf fgh}%
\begin{align}
f_{t}  &  =\frac{f_{s}g_{s}}{fg}f_{ss},\\
g_{t}  &  =\frac{f_{s}g_{s}}{fg}g_{ss},\\
h_{t}  &  =\frac{f_{ss}g_{ss}}{fg}h.
\end{align}
Notice that these equations remain nondegenerate only as long as $f_{s}%
g_{s}/(gf)$ remains strictly positive.

Sometimes, it is more tractable to work with the equations for%
\end{subequations}
\begin{equation}
u=\log f,\qquad v=\log g,\qquad w=\log h.
\end{equation}
These quantities evolve by%

\begin{subequations}
\label{xcf uvw}%
\begin{align}
u_{t}  &  =\alpha\gamma=u_{s}v_{s}u_{ss}+u_{s}^{3}v_{s}\\
v_{t}  &  =\alpha\beta=u_{s}v_{s}v_{ss}+v_{s}^{3}u_{s}\\
w_{t}  &  =\beta\gamma=u_{ss}v_{ss}+v_{s}^{2}u_{ss}+u_{s}^{2}v_{ss}+u_{s}%
^{2}v_{s}^{2},
\end{align}

\end{subequations}
\begin{remark}
Given $w(\cdot,0)$, the evolution of $w$ is determined by $u_{s}$, $v_{s}$,
$u_{ss}$, and $v_{ss}$. So we may (and usually do) suppress its equation in
what follows.
\end{remark}

\begin{remark}
\label{Parabolic}Observe that system~(\ref{xcf fgh}) for $f$ and $g$ ---
equivalently, system~(\ref{xcf uvw}) for $u$ and $v$ --- is strictly
parabolic. This is because our choice to parameterize by arclength $s$ fixes a
gauge and thereby breaks the diffeomorphism invariance of \textsc{xcf}. This
effectively replaces the DeTurck diffeomorphisms in our proof of short-time
existence (Theorem~\ref{ste}, below). This simplification is possible because,
in contrast to the more general situation considered by Buckland \cite{Buck},
the spatial dependence of our metrics is only on the radial coordinate
$r\in\lbrack0,1]$. A similar situation occurs for rotationally invariant Ricci
flow solutions; compare the system for $\varphi$ and $\psi$ considered in
\cite{AK04}, for example.
\end{remark}

Because the solid torus is a manifold-with-boundary, one must prescribe
boundary conditions. We prescribe the Dirichlet boundary conditions given by
\begin{subequations}
\label{boundary conditions}%
\begin{align}
u(x,t)  &  =\log\left(  f(1,0)(1+4t)^{1/4}\right) \\
v(x,t)  &  =\log\left(  g(1,0)(1+4t)^{1/4}\right) \\
w(x,t)  &  =\log\left(  {(1+4t)^{1/4}}\right)  .
\end{align}
As noted above, our choice of boundary conditions are those attained by a
homothetically evolving hyperbolic metric.

\section{Summary of results\label{Summary}}

We now state our results that establish long-time existence and preservation
of negative sectional curvature for the metrics under consideration in this
paper. To streamline the exposition, we postpone the proofs until the next section.

\subsection{Long-time existence and uniqueness}

We first establish short-time existence and uniqueness. Since we are
considering \textsc{xcf} on a manifold-with-boundary, these results do not
immediately follow from \cite{Buck}. So we use the machinery of \cite{ac}. Fix
$\theta\in(0,1)$.
\end{subequations}
\begin{definition}
Given $\tau>0$, let $E_{\tau}$ denote the space
\[
E_{\tau}:=C^{1+\theta}\left(  [0,\tau],C^{0}(\overline{\Omega},C^{N})\right)
\cap L^{\infty}\left(  [0,\tau),C^{2+2\theta}(\overline{\Omega},C^{N})\right)
,
\]
\noindent endowed with the norm
\[
||\vec{X}||_{E_{\tau}}:=\sup_{t\in\lbrack0,\tau]}||\vec{X}||_{C^{0}}%
+[D_{t}\vec{X}]_{C^{\theta}}+\sup_{t\in\lbrack0,\tau]}[D_{x}^{2}\vec
{X}]_{C^{2\theta}}.
\]

\end{definition}

Here, $C^{r+\theta}$ denotes the usual H\"{o}lder norm, and $L^{\infty}$ has
its standard meaning.

\begin{theorem}
\label{ste} There exists $\tau_{0}>0$ such that the fully nonlinear parabolic
system (\ref{xcf uvw}) has a unique solution $(u,v,w)\in E_{\tau_{0}}$.
\end{theorem}


\smallskip

It is not known in general whether solutions to cross curvature flow exist for
all time. However, in the particular case of our rotationally symmetric metric
on the solid torus, we can show \emph{a priori} bounds on the derivatives of
the sectional curvatures.

\begin{theorem}
\label{deriv ests} Let $G$ be a rotationally symmetric metric on the solid
torus given by equation~(\ref{G}). For every $\tau>0$ and $m\in\mathbb{N}$,
there exists a constant $C_{m}$ depending on $m,\tau,$ and $K$ such that if
\[
\beta(s,t),\gamma(s,t)\leq K\mbox{  for all  }s\in M\mbox{  and  }t\in
\lbrack0,\tau],
\]
then
\[
\left\vert \frac{\partial^{m}}{\partial s^{m}}\beta(s,t)\right\vert \leq
\frac{C_{m}}{t^{m}}\mbox{  for all  }s\in M\mbox{  and  }t\in\lbrack0,\tau],
\]
and
\[
\left\vert \frac{\partial^{m}}{\partial s^{m}}\gamma(s,t)\right\vert \leq
\frac{C_{m}}{t^{m}}\mbox{  for all  }s\in M\mbox{  and  }t\in\lbrack0,\tau].
\]

\end{theorem}

We use these estimates together with the fact that the curvatures stay bounded
to show, in the usual way, long-time existence of solutions to
equations~(\ref{xcf fgh}).

\begin{theorem}
\label{long-long-time}The solution $(u,v,w)$ to \textsc{xcf} on the solid
torus exists for all time.
\end{theorem}

\subsection{Curvature estimates}

One can compute the following evolution equations for the sectional
curvatures:
\begin{subequations}
\label{curvev}%
\begin{align}
\alpha_{t}  &  =\alpha\alpha_{ss}+[2\alpha(u_{s}+v_{s})+\beta u_{s}+\gamma
v_{s}]\alpha_{s}+2\alpha(\alpha^{2}-2\beta\gamma)\\
\beta_{t}  &  =\alpha\beta_{ss}+(3\beta u_{s}+\gamma v_{s}-2\alpha u_{s}%
)\beta_{s}+2\beta\lbrack u_{s}^{2}(\alpha-\beta)-\alpha\gamma]\\
\gamma_{t}  &  =\alpha\gamma_{ss}+(3\gamma v_{s}+\beta u_{s}-2\alpha
v_{s})\gamma_{s}+2\gamma\lbrack v_{s}^{2}(\alpha-\gamma)-\alpha\beta].
\end{align}
This system is well behaved as long as $\alpha>0$, since this makes
$\alpha\frac{\partial^{2}}{\partial s^{2}}$ an elliptic operator. In
particular, a maximum principle then applies at interior points. We say
$s(r,t)$ is an \emph{interior point} if $r\in(0,1)$.

Our goal, in the context of \S \ref{Description}, is to use these evolution
equations to show convergence to hyperbolic. Here we collect various estimates
that represent progress toward that goal. Define%
\end{subequations}
\[
K_{0}=\sup_{M\times\{0\}}\{\alpha,\beta,\gamma\}\qquad\text{and}\qquad
L_{0}=\inf_{M\times\{0\}}\{\alpha,\beta,\gamma\}.
\]

\begin{theorem}
\label{curvature estimates}~

\begin{enumerate}
\item For as long as a solution exists, $\alpha>0$. Thus the \textsc{xcf}
operator remains elliptic.

\item The quantities $\alpha,\beta,$ and $\gamma$ are bounded from above by
$K_{0}$.

\item The quantities $\alpha,\beta$, and $\gamma$ are bounded from below by
$L_{0}e^{-4K_{0}^{2}t}$.

\item Negative sectional curvature is preserved.

\item Over the course of the \textsc{xcf} evolution, one has
\[
\alpha\geq\frac{L_{0}}{4K_{0}L_{0}t+1}.
\]

\end{enumerate}
\end{theorem}

\medskip

Finally, we follow \cite{ChowHam} to obtain evidence that our solution
converges to a hyperbolic metric in an integral sense. We use their notation
and define
\begin{equation}
J=\int_{M}(\frac{P}{3}-(\det P)^{\frac{1}{3}})\,\mathrm{dV}\,, \label{J defn}%
\end{equation}
where $P=g_{ij}P^{ij}$. Notice the integrand is nonnegative (by the
arithmetic-geometric mean inequality) and is identically zero if and only if
$g_{ij}$ has constant curvature. Hamilton and Chow's theorem does not directly
apply to our setting, as we have a manifold with boundary. However, we are
able to prove the analogous theorem.

\begin{theorem}
\label{mon of J}Under \textsc{xcf} of a rotationally-symmetric metric on the
solid torus, one has
\begin{equation}
\frac{dJ}{dt}\leq0.
\end{equation}

\end{theorem}

\section{Collected proofs\label{CollectedProofs}}

\subsection{Proofs of long-time existence and uniqueness}

We first use the theory of \cite{ac} for short-time existence and regularity
of solutions to fully nonlinear parabolic systems. Consider the second-order system%

\begin{gather*}
\partial_{t}\vec{X}=F(t,x,\vec{X},D\vec{X},D^{2}\vec{X}%
)\mbox{ for $(t,x)$ in $[0,T]\times \overline{\Omega}$}\\
H(t,x,\vec{X},D\vec{X}%
)=0\mbox{ for $(t,x)$ in $[0,T]\times \partial \Omega$}\\
\vec{X}(0,x)=\phi(x)\mbox{ for $x$ in $\overline{\Omega}$}
\end{gather*}
\noindent where $T>0$, $\Omega$ is a bounded, smooth domain of $\mathbb{R}%
^{n}$, and $F,H$ are smooth $\mathbb{C}^{N}$-valued functions.

The short-time existence theory requires the following hypotheses:

\noindent\textsc{Regularity: }The boundary of $\Omega$ and the functions
$F,H,\phi$ satisfy
\[
\partial\Omega\in C^{2+2\theta},\phi\in C^{2+2\theta}(\overline{\Omega
},\mathbb{C}^{N})
\]%
\[
F\in C^{2}(\Gamma,\mathbb{C}^{N}),
\]%
\[
H\in C^{3}(\Gamma^{\prime N}),(\theta\in(0,1/2)),
\]
\noindent where $\Gamma:=[0,\infty)\times\overline{\Omega}\times\mathbb{C}%
^{N}\times\mathbb{C}^{nN}\times\mathbb{C}^{n^{2}N}$ and $\Gamma^{\prime
}:=[0,\infty)\times\overline{\Omega}\times\mathbb{C}^{N}\times\mathbb{C}^{nN}$.

\noindent\textsc{Compatibility: }The initial datum $\phi$ satisfies
\[
H(0,x,\phi(x),D\phi(x))=0,x\in\partial\Omega.
\]

If these hypotheses are satisfied, one has:

\begin{theorem}
\label{FullyNonlinear}There exists $\tau_{0}>0$ such that the fully nonlinear
parabolic system has a unique solution $\vec{X}\in E_{\tau_{0}}$, where
\[
E_{\tau}:=C^{1+\theta}\left(  [0,\tau],C^{0}(\overline{\Omega},C^{N})\right)
\cap L^{\infty}\left(  [0,\tau),C^{2+2\theta}(\overline{\Omega},C^{N})\right)
,
\]
\noindent with the norm
\[
||\vec{X}||_{E_{\tau}}:=\sup_{t\in\lbrack0,\tau]}||\vec{X}||_{C^{0}}%
+[D_{t}\vec{X}]_{C^{\theta}}+\sup_{t\in\lbrack0,\tau]}[D_{x}^{2}\vec
{X}]_{C^{2\theta}}.
\]

\end{theorem}

\begin{proof}
[Proof of Theorem~\ref{ste}]As explained in Remark~\ref{Parabolic}, the
parameterization by arclength $s$, as defined in (\ref{DefineArcLength}),
makes \textsc{xcf} a strictly parabolic system; furthermore, this
parameterization respects the stated Dirichlet boundary conditions.

Thus we may apply Theorem~\ref{FullyNonlinear} to the system of equations
(\ref{xcf uvw}) with boundary conditions given by (\ref{boundary conditions}).
Clearly, the regularity hypotheses for $F$ and $H$ are satisfied, because our
initial data are $C^{\infty}$. Moreover, one has $u(x,0)=\log{f(1,0)}$,
$v(x,0)=\log{g(1,0)}$, and $w(x,0)=0$ for $x$ on the boundary torus, so that
the compatibility hypothesis is likewise met. Hence Theorem~\ref{ste} follows.
\end{proof}

\medskip

The proof of long-time existence of solutions to the flow is standard once we
obtain \emph{a priori} derivative estimates. So we will first outline the
proof of Theorem~\ref{deriv ests}. We prove this using estimates of
Bernstein--Bando--Shi type.

\begin{proof}
[Sketch of proof of Theorem~\ref{deriv ests}]Define $M_{+}:=t\beta_{s}%
+\lambda_{1}\beta^{2}+\lambda_{2}\gamma^{2}$, where $\lambda_{1}$ and
$\lambda_{2}$ are two constants to be chosen later. (We will see that they
depend only on the initial conditions and the length of the time interval
under consideration.) When one computes the evolution of $M_{+}$, one sees
that it has the following structure:%
\begin{align*}
\frac{\partial M_{+}}{\partial t}  &  =\alpha(M_{+})_{ss}+(M_{+})_{s}%
F_{1}+3t(\beta_{s})^{2}u_{s}+t\beta_{s}\gamma_{s}v_{s}\\
&  +t\beta_{s}F_{2}+t\gamma_{s}F_{3}-2\alpha\lambda_{1}(\beta_{s})^{2}%
-2\alpha\lambda_{2}(\gamma_{s})^{2},
\end{align*}
where $F_{i}=F_{i}(\alpha,\beta,\gamma)$ are polynomials depending on the curvatures.

Now let $\tau>0$ be some real number, and consider the evolution equation on
$[0,\tau]$. Using bounds on the curvatures and multiple applications of
Cauchy--Schwarz, one can show that
\[
\frac{\partial M_{+}}{\partial t}\leq\alpha(M_{+})_{ss}+(M_{+})_{s}%
F_{1}+(\beta_{s})^{2}(F_{2}-2K\lambda_{1})+(\gamma_{s})^{2}(F_{3}%
-2K\lambda_{2})+F_{4},
\]
where $K>0$ is the lower bound for $\alpha$ on $[0,\tau]$, and the $F_{i}$
(which may differ from above) also depend on $\tau$.

Since all of the $F_{i}$ are bounded from above, we choose $\lambda_{1}$ and
$\lambda_{2}$ so that the terms $F_{2}-2K\lambda_{1}$ and $F_{3}-2K\lambda
_{2}$ are both negative. Then we have%
\[
(M_{+})_{t}\leq\alpha(M_{+})_{ss}+(M_{+})_{s}F_{1}+C,
\]
where $C$ is a constant depending only on the initial data. Thus by the
parabolic maximum principle, we have
\[
\sup_{x\in M^{3}}M_{+}(x,t)\leq Ct+D\leq C\tau+D,
\]
for all $t\in\lbrack0,\tau]$, where again $D$ is a constant just depending on
the initial data. Hence, on this bounded time interval,
\[
\beta_{s}\leq\frac{C}{t}.
\]

Notice that in an analogous fashion, we can obtain a lower bound for
$\beta_{s}$ on a bounded time interval by considering $M_{-}:=-t\beta
_{s}-\lambda_{1}\beta^{2}-\lambda_{2}\gamma^{2}$. Thus we obtain our desired
result that $|\beta_{s}|\leq\frac{C}{t}$ on $[0,\tau]$, where $C$ depends only
on the initial data and on $\tau$.

Similarly, to estimate $\gamma_{s}$, one considers $N_{+}:=t\gamma_{s}%
+\lambda_{1}\beta^{2}+\gamma^{2}$ along with $N_{-}:=-t\gamma_{s}-\lambda
_{1}\beta^{2}-\lambda_{2}\gamma^{2}$ and shows that in fact $|\gamma_{s}%
|\leq\frac{C}{t}$, for $C$ as above.

We then use induction to show that higher-order estimates $|\frac{\partial
^{m}\beta}{\partial s^{m}}|,|\frac{\partial^{m}\gamma}{\partial s^{m}}%
|\leq\frac{C}{t^{m}}$ hold on $[0,\tau]$ for any $m>0$.
\end{proof}

In order to show long-time existence of solutions to the system given by
(\ref{xcf fgh}), we will prove a theorem analogous to that for Ricci flow.

\begin{theorem}
\label{lte} Let $G_{0}$ be a metric on the solid torus $M^{3}$ given by
(\ref{G}). Then unnormalized cross curvature flow has a unique solution $G(t)$
such that $G(0)=G_{0}$. This solution exists on a maximal time interval
$[0,T)$. If $T\leq\infty$, then at least one of
\[
\lim_{t\rightarrow T}(\sup_{s\in M^{3}}|\alpha(s,t)|),\qquad\lim_{t\rightarrow
T}(\sup_{s\in M^{3}}|\beta(s,t)|),\qquad\lim_{t\rightarrow T}(\sup_{s\in
M^{3}}|\gamma(s,t)|)
\]
is infinite.
\end{theorem}

\begin{proof}
As for Ricci flow, we prove the contrapositive of this statement. Suppose that
the solution exists on a maximal time interval $[0,T)$, and that there exists
$K>0$ such that $\sup_{0\leq t<T}|\alpha(s,t)|$, $\sup_{0\leq t<T}%
|\beta(s,t)|$, and $\sup_{0\leq t<T}|\gamma(s,t)|\leq K$. The key idea is to
show that $G(s,T)$ is a smooth limit metric on the solid torus of the form
given in (\ref{G}). We can define $f(T)$ and $g(T)$ to be
\begin{align*}
f(s,T)  &  =f(s,\tau)+\int_{\tau}^{T}\alpha\gamma f(s,t)\,\mathrm{dt}\,\\
g(s,T)  &  =g(s,\tau)+\int_{\tau}^{T}\alpha\beta g(s,t)\,\mathrm{dt}\,,
\end{align*}
where $\tau\in\lbrack0,T)$ is arbitrary. Using this formulation to compute
derivatives of $f(T)$ and $g(T)$, one can use Theorem~\ref{deriv ests} to
bound the curvature quantities, thereby showing that both $f(T)$ and $g(T)$
are smooth. It remains to show that the metric $G(T)=f(T)^{2}\,d\mu
^{2}+g(T)^{2}\,d\lambda^{2}+ds^{2}$ extends to a smooth metric on the solid
torus; namely, that $g(T)$ extends to an even function such that $g(0,T)>0$
and that $f(T)$ extends to an odd function with $f_{s}(0,T)=1$. These facts
can be seen from the integral formulation above, if one recalls that the
curvatures extend to even functions. Then $G(T)$ is a smooth metric on the
solid torus, so Theorem~\ref{ste} implies that a solution exists on
$[T,T+\epsilon)$ for some $\epsilon>0$. This contradicts $T$ being maximal.
\end{proof}

\begin{corollary}
[Theorem~\ref{long-long-time}]Let $G_{0}$ be as in (\ref{G}). Then the
solution $G(t)$ to cross curvature flow with $G(0)=G_{0}$ exists for all time.
\end{corollary}

\begin{proof}
By Parts~(2) and (3) of Theorem~\ref{curvature estimates}, we obtain uniform
bounds for $\alpha,\beta,\gamma$. Thus Theorem~\ref{lte} implies $T=\infty$,
to wit, that the flow exists for all time.
\end{proof}

\subsection{Analysis of core conditions}

Recall that one of our requirements on the metric $G$ was that $f=0$ on the
core circle. Then the form in which the evolution equations (\ref{xcf uvw})
are written involves the quantity $u_{s}=f_{s}/f$, which blows up as
$s\searrow0$. Because of this, we treat the core as a special case for the
curvature evolution equations. We do this in a standard way using
l'H\^{o}pital's rule, beginning with the following lemma.

\begin{lemma}
\label{alphabetau_s}On the solid torus, one has%
\[
\lim_{s\rightarrow0}\,u_{s}^{2}(\alpha-\beta)=\lim_{s\rightarrow0}\ \frac
{1}{3}(\beta\gamma-\beta^{2}-\beta_{ss}).
\]

\end{lemma}

\begin{proof}
Writing the expression $u_{s}(\alpha-\beta)$ in terms of $f$ and $g$ and their
derivatives, we obtain
\[
u_{s}^{2}(\alpha-\beta)=\frac{f_{s}^{2}}{f^{2}}\left(  \frac{f_{s}g_{s}}%
{fg}-\frac{g_{ss}}{g}\right)  =\frac{f_{s}^{2}}{g}\left(  \frac{f_{s}%
g_{s}-fg_{s}s}{f^{3}}\right)  .
\]
The quantity outside the parentheses above has a well-defined limit as
$s\rightarrow0$, and both the numerator and denominator of the fraction inside
the parentheses approach $0$ as $s\rightarrow0$. We apply l'H\^{o}pital's rule
three times to obtain the following:
\begin{align*}
\lim_{s\rightarrow0}\,u_{s}^{2}(\alpha-\beta)  &  =\lim_{s\rightarrow0}%
\ \frac{f_{s}}{g}\left(  \frac{f_{ss}g_{s}-fg_{sss}}{3f^{2}}\right) \\
&  =\lim_{s\rightarrow0}\ \frac{1}{g}\left(  \frac{f_{sss}g_{s}+f_{ss}%
g_{ss}-f_{s}g_{sss}-fg_{ssss}}{6f}\right) \\
&  =\lim_{s\rightarrow0}\ \frac{1}{3}\left(  f_{sss}\frac{g_{ss}}{g}%
-\frac{g_{ssss}}{g}\right)  .
\end{align*}
A quick computation of the limit of $\beta_{ss}$ as $s\rightarrow0$ reveals
that
\[
\lim_{s\rightarrow0}\,\beta_{ss}=\lim_{s\rightarrow0}\ \frac{g_{ssss}}%
{g}-\frac{g_{ss}^{2}}{g^{2}},
\]
and this together with the fact that $f_{sss}$ and $\gamma$ have the same
limit as $s\rightarrow0$ yields the conclusion.
\end{proof}

From the result above, the following description of the curvature evolution at
the core follows readily.

\begin{lemma}
\label{curvevcore}At the core, the evolution equations for the curvatures are
as follows:
\begin{align}
&  \alpha_{t}=4\alpha\alpha_{ss}+2\alpha^{3}-4\alpha^{2}\gamma\\
&  \beta_{t}=\frac{4}{3}\alpha\beta_{ss}-\frac{2}{3}\beta^{3}-\frac{4}{3}%
\beta^{2}\gamma\\
&  \gamma_{t}=2\alpha\gamma_{ss}-2\alpha^{2}\gamma.
\end{align}

\end{lemma}

\begin{proof}
We recall that the curvatures all extend past the origin to smooth even
functions of $s$. Using this fact and l'H\^{o}pital's rule, we have
\[
\lim_{s\rightarrow0}u_{s}\alpha_{s}=\lim_{s\rightarrow0}\alpha_{ss},
\]
with similar identities holding for $\beta$ and $\gamma$. The lemma follows
readily from this fact and Lemma~\ref{alphabetau_s}, using the fact that
$\alpha=\beta$ at the core.
\end{proof}

\begin{remark}
It may be noted that since $\alpha=\beta$ at the core, their evolution
equations should be identical. In fact this is true; the missing ingredient is
the fact that
\[
\lim_{s\rightarrow0}\,\alpha_{ss}=\lim_{s\rightarrow0}\ \frac{1}{3}(\beta
_{ss}+2\beta\gamma-2\beta^{2}).
\]
This may be proved by explicitly writing the formula for $\alpha_{ss}$ in
terms of $f$ and $g$ and their derivatives, and using l'H\^{o}pital's rule to
simplify the limit as $s\rightarrow0$, as above.
\end{remark}

\subsection{Proofs of curvature estimates}

Here we provide the proofs of the curvature estimates, using the maximum principle.

\begin{lemma}
As long as a smooth \textsc{xcf} solution exists with $\alpha>0$, one has
$\beta\geq0$ and $\gamma\geq0$.
\end{lemma}

\begin{proof}
This is a standard maximum principle argument, slightly tweaked to accommodate
the core and boundary conditions. We give the proof for $\beta$; the proof for
$\gamma$ is similar. Suppose a smooth solution with $\alpha>0$ exists on the
time interval $[0,\tau]$. Let $C_{0}$ be the maximum attained by the function
$u_{s}^{2}(\alpha-\beta)-\alpha\gamma$ on this time interval. (Recall that
although $u_{s}$ blows up at the origin, $u_{s}^{2}(\alpha-\beta)$ is at least
continuous there by Lemma~\ref{alphabetau_s}.) Define $C_{1}$ to be the
maximum achieved by $-2/3(\beta^{2}+2\beta\gamma)$ at the core on the time
interval $[0,\tau]$, and let $C$ be the maximum of $C_{0}$ and $C_{1}$. Define
$\phi(s,t)=e^{-Ct}\beta(s,t)$. The evolution equation for $\phi$ is
\[
\phi_{t}=\alpha\phi_{ss}+(3\beta u_{s}+\gamma v_{s}-2\alpha u_{s})\phi
_{s}+\phi(u_{s}^{2}(\alpha-\beta)-\alpha\gamma-C)
\]
on the interior, and at the core it is
\[
\phi_{t}=\frac{4}{3}\alpha\phi_{ss}+\phi(-\frac{2}{3}(\beta^{2}+2\beta
\gamma)-C).
\]
At a local minimum for $\phi$ in the interior or at the core, one has
$\phi_{ss}\geq0$ and $\phi_{s}=0$; and from our definition of $C$, it follows
that if $\phi<0$ at such a local minimum, then $\phi_{t}\geq0$ there.

For $\delta>0$, consider the function $\phi+\delta(t+1)$. Initially this
function has all values larger than $0$, since $\beta$ has all values larger
than $0$ initially. If there is a first time in $(0,\tau]$ that $\phi
(s,t)+\delta(t+1)=0$ for some $s$, then the observations above imply that
$(\phi+\delta(t+1))_{t}=\phi_{t}+\delta\geq\delta$ there. (Such a minimum
cannot occur on the boundary, since values of $\beta$ are always positive
there.) But since $t$ is the first time that such a minimum occurs, computing
the time derivative $(\phi+\delta(t+1))_{t}$ from below shows that this
quantity must be less than or equal to $0$, a contradiction. Thus $\phi
+\delta(t+1)$ is positive on $[0,\tau]$, and since $\delta>0$ is arbitrary, so
is $\phi$. But then so is $\beta$, since $\phi$ is a positive multiple of
$\beta$.
\end{proof}

\begin{lemma}
\label{alphalb} Suppose a solution exists for $0\leq t\leq T$. Define
$K=\sup_{M^{3}\times\lbrack0,T]}\{\beta,\gamma\}$. Then for all $t\in
\lbrack0,T]$, one has
\[
\alpha_{\min}(t)\geq\alpha_{\min}(0)e^{-4K^{2}t}.
\]

\end{lemma}

\begin{proof}
Note that $\alpha_{\min}(0)\leq K$, since the curvatures are all equal to $-1$
at the origin at time $t=0$. For $\delta>0$, define the barrier function
$A(t)=\alpha_{\min}(0)e^{-4K^{2}t}-\delta$. Then $\alpha\geq A+\delta$ at
$t=0$. If there is a first time $t\in(0,T]$ such that $\alpha_{\min}(t)=A(t)$,
then $\alpha_{\min}(t)$ is attained either at an interior point or the core.

In the former case, one has
\[
\left\{
\begin{array}
[c]{l}%
\alpha_{t}\leq A^{\prime},\\
\alpha\alpha_{ss}=A\alpha_{ss}\geq0,\\
\alpha_{s}=0.
\end{array}
\right.
\]
By hypothesis $\beta\leq K$ and $\gamma\leq K$. Hence at $(s,t)$, one has
\[
A^{\prime2}(A+\delta)\geq\alpha_{t}>-4\alpha\beta\gamma\geq-4K^{2}A.
\]
This is evidently impossible.

If the minimum occurs at the core, then appealing to the evolution equation
there shows that $\alpha_{t}\geq-4\alpha^{2}\gamma\geq-2A^{2}K$. This yields
\[
A^{\prime2}(A+\delta)\geq\alpha_{t}\geq-2A^{2}K.
\]
Since $K\geq\alpha_{\min}(0)$ and $K\geq1$ we have $K>A$ at $t$. Plugging into
the above inequality yields
\[
A^{\prime2}(A+\delta)\geq-4A^{2}K>-4K^{2}A.
\]
This is plainly a contradiction. Since $\delta$ was arbitrary and $\alpha$ is
controlled on the boundary, the lemma follows.
\end{proof}

\begin{corollary}
[Part~(1) of Theorem~\ref{curvature estimates}]For as long as a solution
exists, $\alpha>0$.
\end{corollary}

This follows immediately from the lemma, and shows that the \textsc{xcf}
operator remains elliptic for as long as the flow exists.

\smallskip

Now let $K_{0}=\sup_{M\times\{0\}}\{\alpha,\beta,\gamma\}$.

\begin{lemma}
\label{betagammaub} Suppose a solution exists for $0\leq t\leq T$. Define
\[
K^{\prime}=\max\ \{K_{0},\sup_{M\times\lbrack0,T]}\alpha\}.
\]
Then with $K$ as in Lemma~\ref{alphalb}, we have $K\leq K^{\prime}$.
\end{lemma}

\begin{proof}
We first consider the case of $\gamma$. The key observation is that at an
interior maximum $(s,t)$ with $\gamma(s,t)\geq\alpha$, we have
\[
\gamma_{t}(s,t)\leq2\beta(v_{s}^{2}(\alpha-\gamma)-\alpha\beta)\leq
-2\alpha\beta\gamma\leq0.
\]
This follows from the evolution equation for $\gamma$ after observing that at
such a point, one has $\alpha\gamma_{ss}\leq0$ and $\gamma_{s}=0$. The
corresponding fact at the core follows from the evolution equations in
Lemma~\ref{curvevcore}.

Now for $\delta>0$, define $\gamma_{\delta}(s,t)=\gamma(s,t)-\delta(t+1)$.
Then $\gamma_{\delta}$ is initially smaller than $K^{\prime}$. If there is a
first time $t$ with $\gamma_{\delta}(s,t)=K^{\prime}$, then by the above we
have $(\gamma_{\delta})_{t}\leq-\delta$ at such a point. But we must have
$(\gamma_{\delta})_{t}\geq0$, a contradiction. Since $\delta>0$ was arbitrary
and $\gamma$ is controlled on the boundary, the lemma follows.

The proof for $\beta$ is analogous.
\end{proof}

Notice that the estimates of Lemmas~\ref{alphalb} and \ref{betagammaub} do not
give \emph{a priori} bounds (in terms of the initial data) on the curvatures,
as they ultimately rely on the upper bound attained by $\alpha$ over the
course of the evolution. We now show that an upper bound is in fact the
quantity $K_{0}$.

\begin{lemma}
\label{alphaub}For as long as the flow exists, $\alpha(s,t)\leq K_{0}$.
\end{lemma}

\begin{proof}
Define $\alpha_{\delta}(s,t)=\alpha(s,t)-\delta(t+1)$. Then $\alpha_{\delta
}(\cdot,0)<K_{0}$. Suppose there is a first time $t_{0}$ at which
$\alpha_{\delta}$ attains the value $K_{0}$. We claim that the value
$K_{0}+\delta(t_{0}+1)$ attained by $\alpha$ at $\delta_{0}$ is the maximum
value it attains on the interval $[0,t_{0}]$. To see this, suppose there were
some $t<t_{0}$ such that $\alpha(s,t)=K_{0}+\delta(t_{0}+1)$ for some $s$.
Then $\alpha_{\delta}(s,t)=K_{0}+\delta(t_{0}-t)>K_{0}$, a contradiction to
our hypothesis that $t_{0}$ is the first time $\alpha_{\delta}$ attains the
value $K_{0}$. Thus by Lemma~\ref{betagammaub}, the values of $\beta$ and
$\gamma$ are bounded above by $K_{0}+\delta(t_{0}+1)$ on the interval
$[0,t_{0}]$. In particular (this is the important point), at the maximum of
$\alpha$ at time $t_{0}$, we have $\alpha\geq\beta$ and $\alpha\geq\gamma$.
Appealing to the evolution equation for $\alpha$ recorded in (\ref{curvev}),
one may obtain a contradiction in the usual way if the quantity $\alpha
^{2}-2\beta\gamma<0$. The difficult case (and the reason we have not given
\emph{a priori }bounds to this point) is when this inequality does not hold.
To deal with this, we use the fact that $\alpha=u_{s}v_{s}$ to rewrite the
evolution equation for $\alpha$. Recall that $\alpha_{s}=\beta u_{s}+\gamma
v_{s}-\alpha(u_{s}+v_{s})$. Using this, we rewrite the evolution equation for
$\alpha$ as follows:
\begin{equation}%
\begin{array}
[c]{rl}%
\alpha_{t}= & \alpha\alpha_{ss}+[\beta u_{s}+\gamma v_{s}]\alpha_{s}%
+2\alpha\lbrack(u_{s}+v_{s})\alpha_{s}+\alpha^{2}-2\beta\gamma]\\
= & \alpha\alpha_{ss}+[\beta u_{s}+\gamma v_{s}]\alpha_{s}+2\alpha\lbrack
u_{s}^{2}(\beta-\alpha)+v_{s}^{2}(\gamma-\alpha)+\alpha\beta+\alpha
\gamma-\alpha^{2}-2\beta\gamma]\\
= & \alpha\alpha_{ss}+[\beta u_{s}+\gamma v_{s}]\alpha_{s}-2\alpha\lbrack
u_{s}^{2}(\alpha-\beta)+v_{s}^{2}(\alpha-\gamma)+(\alpha-\beta)(\alpha
-\gamma)+\beta\gamma].
\end{array}
\label{alphaev}%
\end{equation}
At the maximum for $\alpha$ at time $t_{0}$, we have already observed that
$\beta\leq\alpha$ and $\gamma\leq\alpha$. Therefore the zero-order term of the
differential equation (\ref{alphaev}) is nonpositive, and a contradiction is
obtained using the maximum principle in the standard way. Since $\delta>0$ was
arbitrary, the Lemma is proved if the maximum occurs in the interior. If the
maximum occurs at the core, we note since $\alpha=\beta$ there and $\beta\leq
K_{0}+\delta(t_{0}+1)$, this must also be a local maximum for $\beta$. Using
the evolution equation for $\beta$ at the core, we find that $\beta_{t}%
=\alpha_{t}\leq0$, also contradiction. Hence the lemma is proved.
\end{proof}

The following theorem is an immediate corollary of Lemma~\ref{betagammaub} and
Lemma~\ref{alphaub}.

\begin{theorem}
[Part~(2) of Theorem~\ref{curvature estimates}]\label{univub} For as long as
the flow exists, $\alpha$, $\beta$, and $\gamma$ are bounded above by $K_{0}$.
\end{theorem}

The universal upper bound of Theorem~\ref{univub} may be used to give a
universal lower bound for the sectional curvatures, as in Lemma~\ref{alphalb}.

\begin{theorem}
[Part~(3) of Theorem~\ref{curvature estimates}]\label{univlb} Let $L_{0}%
=\inf_{M\times\{0\}}\{\,\alpha,\beta,\gamma\,\}$. For as long as the flow
exists, $\alpha$, $\beta$, and $\gamma$ are bounded below by $L_{0}%
e^{-4K_{0}^{2}t}$.
\end{theorem}

\begin{proof}
The case of $\alpha$ follows immediately from Lemma~\ref{alphalb}, after
noting that Theorem~\ref{univub} implies that the constant $K$ in the Lemma is
less than or equal to $K_{0}$, and that $\alpha_{\min}(0)\geq L_{0}$.

We next address the case of $\gamma$. As in the proof of Lemma \ref{alphalb},
for $\delta>0$ we define a barrier function $A(t)=L_{0}e^{-4K_{0}^{2}t}%
-\delta$. Then $\gamma>A$ at $t=0$. If there is a first time $t>0$ such that
$\gamma_{\min}(t)=A(t)$, then at such a point one has
\[
\left\{
\begin{array}
[c]{l}%
\gamma_{t}\leq A^{\prime},\\
\alpha\gamma_{ss}\geq0,\\
\gamma_{s}=0.
\end{array}
\right.
\]
Note that at this point, one has $\alpha>\gamma$, since the inequality holds
for $\alpha$. If the minimum occurs in the interior, then by the evolution
equation satisfied by $\gamma$ there and the fact that $\alpha\leq K_{0}$ and
$\beta\leq K_{0}$, we have
\[
A^{\prime}=-4K_{0}^{2}(A+\delta)\geq\gamma_{t}>-2\alpha\beta\gamma\geq
-2K_{0}^{2}A.
\]
This is a contradiction. If the minimum occurs at the core, appealing to the
evolution equation there (see Lemma~\ref{curvevcore}) yields
\[
A^{\prime}=-4K_{0}^{2}(A+\delta)\geq\gamma_{t}>-2\alpha^{2}\gamma\geq
-2K_{0}^{2}A,
\]
again yielding a contradiction. Since $\delta>0$ is arbitrary the result is
proved for $\gamma$.

The case of $\beta$ is analogous.
\end{proof}

Theorems~\ref{univub} and \ref{univlb} combine to give bounds for the
sectional curvatures that hold for all positive times.

\smallskip

Theorem\ \ref{univlb} has another important consequence:

\begin{corollary}
[Part~(4) of Theorem~\ref{curvature estimates}]For as long as a the solution
of \textsc{xcf} exists, negative sectional curvature is preserved.
\end{corollary}

\smallskip

We can slightly improve the lower bound for $\alpha$.

\begin{lemma}
The evolution equation for $\alpha$ in the interior may be written in the
following forms:
\begin{subequations}
\label{alphaevs}%
\begin{align}
&  \alpha_{t}=\alpha\alpha_{ss}+[\beta u_{s}+\gamma v_{s}+2\alpha(u_{s}%
+v_{s})]\alpha_{s}+2\alpha(\alpha^{2}-2\beta\gamma)\\
&  \alpha_{t}=\alpha\alpha_{ss}+[\gamma v_{s}-3\beta u_{s}+2\alpha(u_{s}%
+v_{s})]\alpha_{s}+4u_{s}^{2}\beta(\beta-\alpha)-2\alpha^{2}(2\beta-\alpha)\\
&  \alpha_{t}=\alpha\alpha_{ss}+[\beta u_{s}-3\gamma v_{s}+2\alpha(u_{s}%
+v_{s})]\alpha_{s}+4v_{s}^{2}\gamma(\gamma-\alpha)-2\alpha^{2}(2\gamma
-\alpha)\\
&  \alpha_{t}=\alpha\alpha_{ss}+[\beta u_{s}+\gamma v_{s}]\alpha_{s}%
-2\alpha\lbrack u_{s}^{2}(\alpha-\beta)+v_{s}^{2}(\alpha-\gamma)+(\alpha
-\beta)(\alpha-\gamma)+\beta\gamma].
\end{align}

\end{subequations}
\end{lemma}

\begin{proof}
Equations (a) and (d) above were previously obtained; they are equations
(\ref{curvev}a) and (\ref{alphaev}), respectively. Equation (b) is obtained
from (a) by separating off a factor of $4\beta u_{s}\alpha_{s}$, so that one
obtains
\[
\alpha_{t}=\alpha\alpha_{ss}+[\gamma v_{s}-3\beta u_{s}+2\alpha(u_{s}%
+v_{s})]\alpha_{s}+4\beta u_{s}\alpha_{s}+2\alpha^{3}-4\alpha\beta\gamma.
\]
It is easily computed that $\beta u_{s}\alpha_{s}=\beta u_{s}^{2}(\beta
-\alpha)+\alpha\beta\gamma-\alpha^{2}\beta$. Substituting this in the equation
above and simplifying yields the result. Equation (c) is obtained analogously,
but by separating off a factor of $4\gamma v_{s}\alpha_{s}$ instead of $4\beta
u_{s}\alpha_{s}$.
\end{proof}

We use the new equations below to improve the lower bound on the decay of
$\alpha$ from exponential to polynomial. Recall that $K_{0}$ is the supremum
of the curvatures of the initial metric, and that $L_{0}$ is the infimum.

\begin{theorem}
[Part~(5) of Theorem~\ref{curvature estimates}]Over the course of the
evolution, one has
\[
\alpha\geq\frac{L_{0}}{4K_{0}L_{0}t+1}.
\]

\end{theorem}

\begin{proof}
For $\delta>0$, we define the barrier function
\[
A(t)=\frac{L_{0}}{4K_{0}L_{0}t+1}-\delta,
\]
and note that $\alpha(\cdot,0)>A(0)$. Suppose there is a first time $t>0$ at
which $\alpha(s,t)=A(t)$ for some $s$. If this point $s$ is in the interior,
the fact that it is a minimum for $\alpha$ implies that $\alpha_{ss}\geq0$ and
$\alpha_{s}=0$ there. Furthermore, we must have
\[
\alpha_{t}\leq A^{\prime}(t)=-4K_{0}(A+\delta)^{2}.
\]
At this point the analysis breaks up into three cases depending on the
relationship of $\alpha$ with the other curvatures.

Suppose first that $\alpha(s,t)<\beta(s,t)$. Then using
equation~(\ref{alphaevs}a), we see from the above that
\[
-4K_{0}(A+\delta)^{2}\geq\alpha_{t}=4u_{s}^{2}\beta(\beta-\alpha)-2\alpha
^{2}(2\beta-\alpha)>-4K_{0}\alpha^{2}=-4K_{0}A^{2}.
\]
This is a contradiction.

If $\alpha(s,t)<\gamma(s,t)$, an analogous analysis using
equation~(\ref{alphaevs}b) yields a contradiction in the same way.

It remains to consider the case that $\alpha\geq\beta,\gamma$. In this case,
we use the original evolution equation~(\ref{curvev}) for $\alpha$. This
yields
\[
-4K_{0}(A+\delta)^{2}\geq\alpha_{t}=2\alpha(\alpha^{2}-2\beta\gamma
)\geq-2\alpha^{3}\geq-2K_{0}A^{2},
\]
and again a contradiction is obtained.

Finally, if the minimum occurs at the core, appealing to the evolution
equation there (see Lemma~\ref{curvevcore}) yields
\[
\alpha_{t}=-2\alpha^{2}(2\gamma-\alpha)\geq-4K_{0}\alpha^{2},
\]
and an analysis similar to the above yields a contradiction. Since $\delta>0$
was arbitrary, the result is proved.
\end{proof}

\subsection{Integral convergence to hyperbolic}

Here we prove Theorem~\ref{mon of J}. Recall that we defined
\[
J=\int_{M}(\frac{P}{3}-(\det P)^{\frac{1}{3}})\,\mathrm{dV}%
\]
in (\ref{J defn}), where $P$ here denotes the trace of the Einstein tensor.
Notice that the integrand is nonnegative (by the arithmetic-geometric mean
inequality) and is identically zero if and only if the metric has constant curvature.

The only difficulty is that we are considering a manifold-with-boundary.
Because the proof of monotonicity in \cite{ChowHam} relies on integration by
parts, we must be able to control the boundary terms that arise in our
situation. In what follows, we adopt their notation and use the following result.

\begin{lemma}
[Chow--Hamilton]The evolution of the Einstein tensor under \textsc{xcf} is
\begin{equation}
\frac{\partial}{\partial t}P^{ij}=\nabla_{k}\nabla_{l}(P^{kl}P^{ij}%
-P^{ik}P^{jl})-\det Pg^{ij}-XP^{ij}, \label{ev of P}%
\end{equation}
where $X=g^{ij}X_{ij}$ is the trace of the cross curvature tensor.
\end{lemma}

\begin{proof}
See \cite{ChowHam}.
\end{proof}

\begin{lemma}
For every smooth $\phi$ defined on our solid torus solution, one has
\begin{equation}
\frac{d}{dt}\int_{0}^{s_{i}(t)}\phi\,\mathrm{ds}\,=\int_{0}^{s_{1}(t)}%
(\phi_{t}+\beta\gamma)\,\mathrm{ds}\,. \label{torus integration}%
\end{equation}

\end{lemma}

\begin{proof}
Straightforward computation.
\end{proof}

We now prove the final result of this paper.

\begin{proof}
[Proof of Theorem~\ref{mon of J}]We begin by computing%
\begin{align*}
\frac{d}{dt}\int_{0}^{s_{1}(t)}P\,\mathrm{ds}\,  &  =\int_{0}^{s_{1}(t)}%
(P_{t}+\beta\gamma P)\,\mathrm{ds}\,\\
&  =\int_{0}^{s_{1}(t)}\left\{
\begin{array}
[c]{c}%
g_{ij}(\nabla_{k}\nabla_{l}(P^{kl}P^{ij}-P^{ik}P^{jl}))\\
-3\det P-XP+2X_{ij}P^{ij}+\beta\gamma P
\end{array}
\right\}  \,\mathrm{ds}\,\\
&  =\int_{0}^{s_{1}(t)}(3\det P-\alpha\beta P-\alpha\gamma P)\,\mathrm{ds}%
\,+\left.  \langle\frac{\partial}{\partial s},P^{kl}\nabla_{l}P-P^{jl}%
\nabla_{l}P_{j}^{k}\rangle\right\vert _{0}^{s_{1}(t)}.
\end{align*}
In terms of our metric, the boundary term becomes
\[
\left.  \left[  \alpha(\beta_{s}+\gamma_{s})-\frac{\beta f_{s}(\alpha-\beta
)}{f}-\frac{\gamma g_{s}(\alpha-\gamma)}{g}\right]  \right\vert _{0}%
^{s_{1}(t)}.
\]
Let $V_{ij}$ denote the inverse of $P^{ij}$. Then%
\begin{align*}
\frac{d}{dt}\int_{0}^{s_{1}(t)}(\det P)^{\frac{1}{3}}\,\mathrm{ds}\,  &
=\int_{0}^{s_{1}(t)}(\partial_{t}(\det P)^{\frac{1}{3}}+\beta\gamma(\det
P)^{\frac{1}{3}})\,\mathrm{ds}\,\\
&  =\int_{0}^{s_{1}(t)}(\det P)^{\frac{1}{3}}(\frac{1}{3}V_{ij}\nabla
_{k}\nabla_{l}(P^{kl}P^{ij}-P^{ik}P^{jl}))\,\mathrm{ds}\,\\
&  +\int_{0}^{s_{1}(t)}(\det P)^{\frac{1}{3}}(\frac{1}{3}V_{ij}(-\det
Pg^{ij}-XP^{ij})+2X)+\beta\gamma)\,\mathrm{ds}\,\\
&  =:I_{1}+I_{2}.
\end{align*}
Let us first consider $I_{2}$. Because $V=(\det P)^{-1}X$, one has
\begin{multline*}
\int_{0}^{s_{1}(t)}(\det P)^{\frac{1}{3}}(\frac{1}{3}V_{ij}(-\det
Pg^{ij}-XP^{ij})+2X)+\beta\gamma)\,\mathrm{ds}\,\\
=\int_{0}^{s_{1}(t)}(\det P)^{\frac{1}{3}}(\frac{1}{3}X-\alpha\gamma
-\alpha\beta)\,\mathrm{ds}\,.
\end{multline*}
Now we do integration by parts on $I_{1}$ to obtain%
\begin{align*}
I_{1}  &  =-\frac{1}{2}\int_{0}^{s_{1}(t)}\nabla_{k}((\det P)^{1/3}%
V_{ij})(P^{kl}\nabla_{l}P^{ij}-P^{jl}\nabla_{l}P^{ik})\,\mathrm{ds}\,\\
&  +\left.  \left\langle (\det P)^{1/3}V_{ij}\otimes\frac{\partial}{\partial
s},P^{kl}\nabla_{l}P^{ij}-P^{jl}\nabla_{l}P^{ik}\right\rangle \right\vert
_{0}^{s_{1}(t)}\\
&  =\frac{1}{3}\int_{0}^{s_{1}(t)}\left(  \frac{1}{2}\left\vert E^{ijk}%
-E^{jik}\right\vert _{V}^{2}+\frac{1}{6}|T^{i}|_{V}^{2}\right)  (\det
P)^{1/3}\,\mathrm{ds}\,\\
&  +\left.  \left\langle (\det P)^{1/3}V_{ij}\otimes\frac{\partial}{\partial
s},P^{kl}\nabla_{l}P^{ij}-P^{jl}\nabla_{l}P^{ik}\right\rangle \right\vert
_{0}^{s_{1}(t)},
\end{align*}
where $E^{ijk}$, \emph{et cetera, }have the same meanings as in \cite{ChowHam}.

A computation shows that the boundary term is
\begin{multline*}
\left.  \langle(\det P)^{\frac{1}{3}}V_{ij}\otimes\frac{\partial}{\partial
s},P^{kl}\nabla_{l}P^{ij}-P^{jl}\nabla_{l}P^{ik}\rangle\right\vert _{0}%
^{s_{1}(t)}=\left.  (\det P)^{\frac{1}{3}}V_{ij}T^{1ij}\right\vert _{0}%
^{s_{1}(t)}\\
=\left.  (\alpha\beta\gamma)^{\frac{1}{3}}(\frac{\alpha}{\beta}\beta_{s}%
+\frac{\alpha}{\gamma}\gamma_{s}+\frac{f_{s}}{f}(\beta-\alpha)+\frac{g_{s}}%
{g}(\gamma-\alpha)\right\vert _{0}^{s_{1}(t)}.
\end{multline*}

Now we collect all of the terms above to see that
\begin{align*}
\frac{d}{dt}J  &  =\int_{0}^{s_{1}(t)}(\det P-\frac{1}{3}\alpha\beta
P-\frac{1}{3}\alpha\gamma P-(\det P)^{\frac{1}{3}}(\frac{X}{3}-\alpha
\gamma-\alpha\beta)\,\mathrm{ds}\\
&  =-\frac{1}{3}\int_{0}^{s_{1}(t)}(\frac{1}{2}|E^{ijk}-E^{jik}|_{V}^{2}%
+\frac{1}{6}|T^{i}|_{V}^{2})(\det P)^{\frac{1}{3}}\,\mathrm{ds}\\
&  +\left.  \alpha(\beta_{s}+\gamma_{s})-\frac{\beta f_{s}(\alpha-\beta)}%
{f}-\frac{\gamma g_{s}(\alpha-\gamma)}{g}\right\vert _{0}^{s_{1}(t)}\\
&  -\left.  (\alpha\beta\gamma)^{\frac{1}{3}}(\frac{\alpha}{\beta}\beta
_{s}+\frac{\alpha}{\gamma}\gamma_{s}+\frac{f_{s}}{f}(\beta-\alpha)+\frac
{g_{s}}{g}(\gamma-\alpha)\right\vert _{0}^{s_{1}(t)}.
\end{align*}
Hence%
\begin{align*}
\frac{d}{dt}J  &  \leq-\int_{0}^{s_{1}(t)}(\det P)^{\frac{1}{3}}(\frac{X}%
{3}-(\det h)^{\frac{1}{3}})\,\mathrm{ds}\\
&  -\int_{0}^{s_{1}(t)}\alpha\beta(\frac{P}{3}-(\det P)^{\frac{1}{3}%
})\,\mathrm{ds}\\
&  -\int_{0}^{s_{1}(t)}\alpha\gamma(\frac{P}{3}-(\det P)^{\frac{1}{3}%
})\,\mathrm{ds}\\
&  +\left\{
\begin{array}
[c]{c}%
\alpha(\beta_{s}+\gamma_{s})-\frac{\beta f_{s}(\alpha-\beta)}{f}-\frac{\gamma
g_{s}(\alpha-\gamma)}{g}\\
-(\alpha\beta\gamma)^{\frac{1}{3}}(\frac{\alpha}{\beta}\beta_{s}+\frac{\alpha
}{\gamma}\gamma_{s}+\frac{f_{s}}{f}(\beta-\alpha)+\frac{g_{s}}{g}%
(\gamma-\alpha)
\end{array}
\right\}  (s_{1})\\
&  +(\alpha\beta\gamma)^{\frac{1}{3}}\left\{
\begin{array}
[c]{c}%
(\frac{\alpha}{\beta}\beta_{s}+\frac{\alpha}{\gamma}\gamma_{s}+\frac{f_{s}}%
{f}(\beta-\alpha)+\frac{g_{s}}{g}(\gamma-\alpha)\\
-\alpha(\beta_{s}+\gamma_{s})-\frac{\beta f_{s}(\alpha-\beta)}{f}-\frac{\gamma
g_{s}(\alpha-\gamma)}{g}%
\end{array}
\right\}  (0).
\end{align*}
Recall that all of the curvatures are equal at the outer boundary. At the
core, $\alpha=\beta$ and all of the curvatures as well as $g$ extend to even
functions. Thus all of the boundary terms cancel. The result follows.
\end{proof}

\end{document}